\documentclass[reqno,12pt]{amsart}
\usepackage[a4paper,bindingoffset=0.5cm,left=2cm,right=2cm,top=2.5cm,bottom=2cm,footskip=.8cm]{geometry}

\usepackage{subcaption}
\usepackage{mathrsfs}
\usepackage{graphicx} 
\usepackage{amsmath, amssymb, amsthm, amsfonts, tikz} 
\usepackage[ruled,vlined]{algorithm2e}
\usepackage{tikz}
\usetikzlibrary{positioning}
\usepackage{hyperref}
\hypersetup{colorlinks=true, linkcolor=blue, citecolor=blue,  filecolor=blue, urlcolor=blue}
\newcommand{\bs}{\boldsymbol}

\newcommand{\E}{\mathbb{E}}

\usepackage{graphicx} 
\usepackage{amsmath,bm,bbm,amsthm, amssymb}
\usepackage{fullpage}
\usepackage{color}
\usepackage{comment}
\theoremstyle{plain}
\newtheorem{theorem}{Theorem}
\newtheorem{proposition}[theorem]{Proposition}

\newtheorem{lemma}[theorem]{Lemma}

\newtheorem{remark}[theorem]{Remark}
\newtheorem{example}[theorem]{Example}

\def\E{\mathbb E}

\def\one{\mathbf 1}
\def\vp{\varphi}
\def\R{\mathbb R}
\def\N{\mathbb N}

\def\d{\mathrm d}
\def\cum{\text{cum}}
\def\xx{\mathbf x}

\def\ov{\overline}
\def\HMAX{^{\rm +}}
\def\HMIN{^{\rm -}}
\def\td{\text{\rm d}}

\newcommand{\netheo}[1]{{Theorem \ref{#1}}}

\newcommand{\proofprop}[1]{\textsc{\bf Proof of Proposition} \ref{#1}:}

\newcommand{\abs}[1]{\left\lvert #1 \right\rvert}

\newcommand{\vk}[1]{\boldsymbol{#1}}
\newcommand{\pk}[1]{\mathbb{P} \left\{ #1 \right \} }
\newcommand{\s}{\sigma}

\title{Functional central limit theorem for subgraph counts in a dynamic random connection model}

\author{Rajat Subhra Hazra}

\author{Nikolai Kriukov}

\author{Michel Mandjes}

\author{Moritz Otto}
\begin{document}

\begin{abstract}
We prove a functional central limit theorem for subgraph counts {in} a dynamic version of the random connection model. To establish tightness, we develop a dynamic extension of the cumulant method.

\bigskip

\noindent {\sc Affiliations.}
Rajat Subhra Hazra is with Mathematical Institute, Leiden University, 
The Netherlands. \url{r.s.hazra@math.leidenuniv.nl}

\medskip

\noindent Nikolai Kriukov is with Korteweg-de Vries Institute for Mathematics, University of Amsterdam, The Netherlands. \url{n.kriukov@uva.nl}

\medskip

\noindent Michel Mandjes is with Mathematical Institute, Leiden University, 
The Netherlands, and Korteweg-de Vries Institute for Mathematics, University of Amsterdam, The Netherlands. \url{m.r.h.mandjes@uva.nl}

\medskip

\noindent
Moritz Otto is with Mathematical Institute, Leiden University, 
The Netherlands.\\ \url{m.f.p.otto@math.leidenuniv.nl}


\medskip
\noindent{\sc MSC2020 subject classifications.} 05C80, 60D05, 60G55.

\bigskip

\noindent  {\sc Acknowledgments.}
The research of Kriukov, Mandjes and Otto was supported by the European Union’s Horizon 2020 research and innovation programme under the Marie Sklodowska-Curie grant agreement no.\ 945045, and by the NWO Gravitation project N{\sc etworks} under grant agreement no.\ 024.002.003. \includegraphics[height=1em]{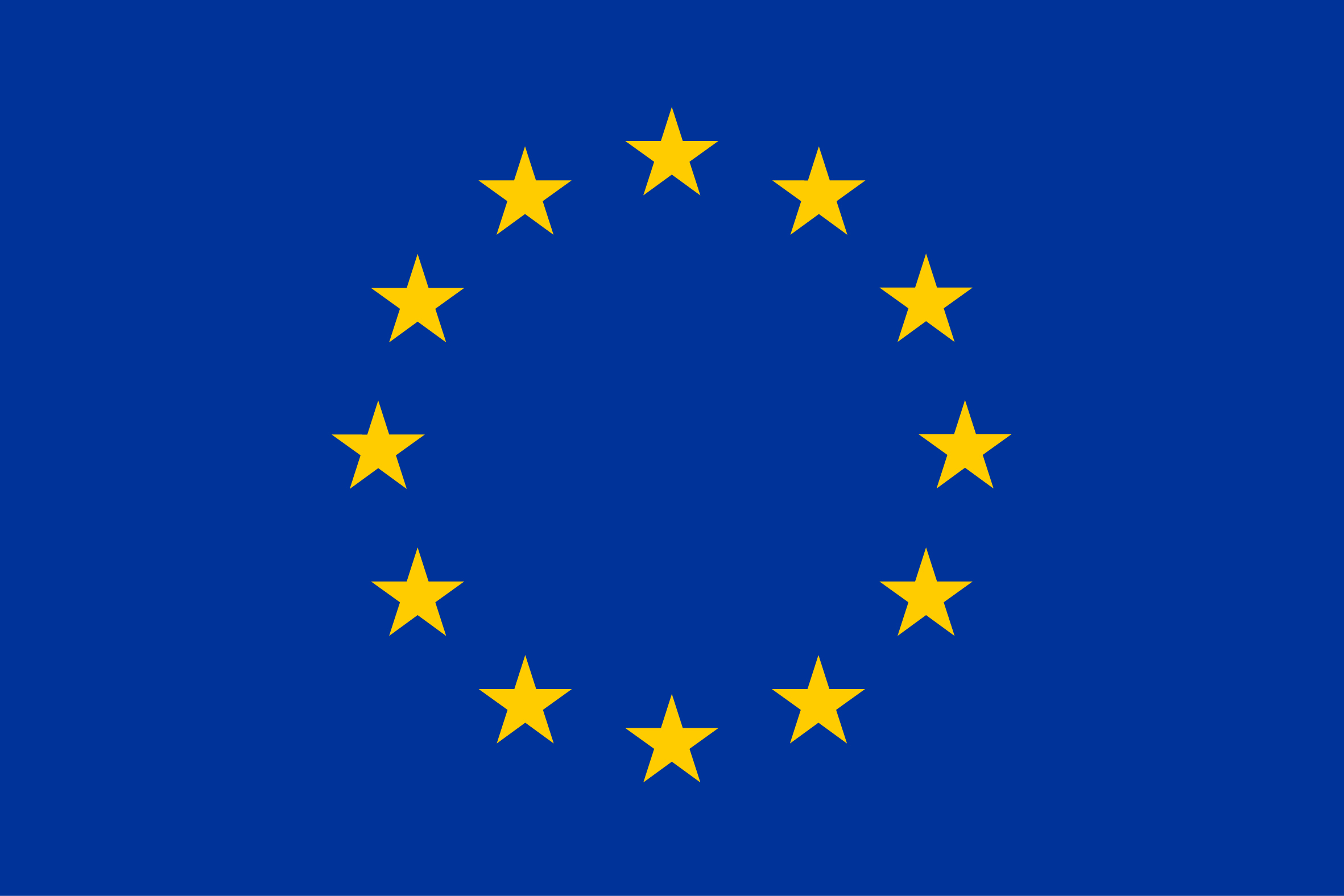} 

\bigskip

\end{abstract}

\date{\today}

\maketitle

\newcommand{\ABS}[1]{\left(#1\right)} 
\newcommand{\veps}{\varepsilon} 

\section{Introduction}

The {\it random connection model} (RCM) is a fundamental stochastic geometric model used to describe spatial networks in which nodes are connected with a probability that depends on their mutual distance. More formally, given a homogeneous Poisson point process in $\mathbb{R}^d$, two points $x$ and $y$ are connected independently with probability that depends on the distance between $x$ and $y$. The RCM generalizes the classical {\em random geometric graph} (or Gilbert graph) by allowing soft (probabilistic) connections instead of deterministic ones. It has been widely studied in the context of percolation theory, continuum random graphs, and wireless networks (see e.g., Meester \& Roy \cite{MeesterRoy1996}, Penrose \cite{Penrose1991}, and Roy \cite{Roy2011}).

In the study of the RCM, an important branch of the literature concerns {\it subgraph counts} where one investigates the number of occurrences of a fixed finite graph $H$ (such as a triangle or a star) as an induced or embedded subgraph of the RCM. These counts not only capture local connectivity patterns but also play a key role in understanding the higher-order structure and limiting behavior of the model. Subgraph counts have been studied using tools from Poisson point process theory, cumulants, Stein’s method, and U-statistics (see, e.g., Penrose \cite{Penrose2003}, Schulte  \& Th\"ale \cite{chaos}, Liu \& Privault \cite{liu2024normal} and Heerten {\it et al.} \cite{heerten}). In particular, asymptotic normality for these counts provide insights into scaling limits.

\medskip

A major trend in the random graph literature is the increasing focus on {\it dynamically evolving networks}, in which edges and/or vertices change stochastically over time to better reflect the behavior of real-world systems such as communication, social, or biological networks. These so-called {\it temporal random graphs}, as described by e.g.\ Holme \& Saram\"aki \cite{HolmeSaramaki2012}, capture both the structural and dynamic complexity of such systems and give rise to new mathematical challenges and insights. This shift reflects a growing interest in understanding not only static connectivity properties, but also how network behavior evolves over time. In principle, any property studied in the static setting can be extended to the dynamic context. For instance, the sample-path large deviation principle developed in Braunsteins {\it et al.} \cite{Braunsteins2023} generalizes the foundational result for the static Erd\H{o}s--R\'enyi graph by Chatterjee \& Varadhan \cite{ChatterjeeVaradhan2011}, while the maximum eigenvalue process studied by Hazra {\it et al.} \cite{hkm25a} generalizes the static result from Erd\H{o}s {\it et al.} \cite{ErdosKnowlesYauYin2013}.

\medskip

In this paper, we consider a dynamic variant of the RCM, in which point locations and (potential) edges are sampled once and remain fixed, but where each point alternates between being active and inactive over time. Only edges between pairs of simultaneously active points are included into the graph. This setting naturally gives rise to a multivariate subgraph count process, tracking the number of occurrences of various subgraphs (such as edges, triangles, or stars) over time. The main result of the paper is a functional central limit theorem (CLT) for a centered and normalized version of this multivariate subgraph count process. As an application of this result, we present a functional CLT for the clustering coefficient process in the dynamic RCM. 

In Section \ref{sec:main} we formally introduce our model and state our main result. Our proofs build on the recently developed cumulant method, as introduced in a series of recent works (see e.g.\ Last {\it et al.} \cite{last2014moments}, Schulte \& Th\"ale \cite{schulte2016cumulants, chaos}). This method establishes a connection between Poisson U-statistics and a specific class of partitions, allowing their cumulants to be expressed in terms of certain tensor products (see Section \ref{sec:prelim}). After having determined the candidate mean and covariances to feature in the functional CLT (Section \ref{section:exp_and_cov}),  the fundamental challenge is to extend the use of cumulants beyond establishing finite-dimensional convergence (Section \ref{section:fdd_convergence}) to the functional setting.  We do so by effectively applying them to prove tightness (Section \ref{section_tightness}), thereby illustrating the broad applicability of this approach.

\section{Model and main results}\label{sec:main}

\subsection{Model description}
Denote
$W := [-\tfrac{1}{2}, \tfrac{1}{2}]^d,$
and let $\eta_n$ be a Poisson point process on
$W \times \mathbb{D}([0,T],\{0,1\}),$ with intensity measure $n\, \d x \otimes \mathbb Q$. Here $T > 0$ is the time horizon, and $\mathbb{Q}$ denotes the law of a two-state Markov jump process on $\{0,1\}$ with exponential waiting times, which jumps from 0 to 1 at rate $\mu$ and from 1 to 0 at rate $\lambda$, started at its stationary distribution. In order to avoid potential boundary effects, throughout this paper we work with the torus metric on $W$, i.e., we identify opposite faces of $W$ (i.e., the space `wraps around' like a torus). We write $P=(X,(A(t))_{t \in [0,T]})$ for a canonical element in $\eta_n$, where $X \in W$ is the location of $P$ and $A(t)$ is its status process at time $t \in [0,T]$.  

Our model is a random connection model (RCM) in which the underlying graph's vertices alternate between being present and absent as follows. Let $\mu>0,\,\lambda>0$ be model parameters. At time $t=0$, each of the points in $\eta_n$ is {\em active} independently with probability $\varrho :={\mu}/{(\mu+\lambda)}$, and {\em inactive} else. As time goes on, the points in $\eta_n$  independently alternate between being active and inactive: inactive (active) points become active (inactive) after an exponentially distributed time with rate $\mu$ ($\lambda$, respectively). 
Note that the expected number of active Poisson points is $\varrho\, n$ for all $t\ge 0$.

Let $\mathcal X \subset W$ be locally finite. In the RCM on $\mathcal X$, between any pair of distinct points $X_i, X_j \in \mathcal X$ we place a (potential) edge independently with probability $\vp_n(X_i - X_j)$, where
\begin{align}
	\label{eq:prob}
	\vp_n({x}):=	\vp\bigg(\frac{\|{x}\|^d}{\nu_n}\bigg).
\end{align}
Here, $(\nu_n)_{n \ge 1}$ is a positive real-valued sequence with $\nu_n \to 0$ as $n \to \infty$, and $\vp: [0,\infty) \to [0, 1]$  (typically referred to as the \emph{profile function}) is assumed to be non-increasing and to satisfy the normalization $\int \vp(t) \,\d t = 1$.

The edges are sampled once and for all, but they appear in the RCM only when both adjacent nodes are simultaneously active (see Figure \ref{fig:drcm}). The resulting dynamic random graph process we denote by $\mathrm{RCM}(\mathcal X,\vp_n)$.

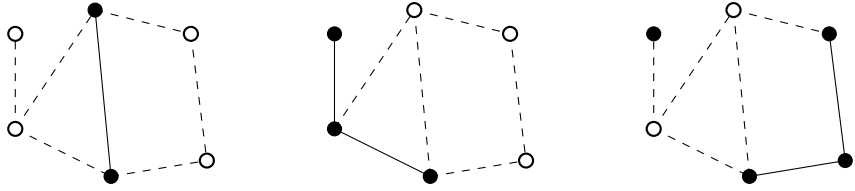
\begin{figure}
\begin{center}
\begin{tikzpicture}[scale=1.05]

\def\nodes{
  \coordinate (A) at (0,0);
  \coordinate (B) at (1,1.5);
  \coordinate (C) at (2.2,1.2);
  \coordinate (D) at (1.2,-0.6);
  \coordinate (E) at (2.4,-0.4);
  \coordinate (F) at (0,1.2);
}

\begin{scope}[xshift=0cm]
  \nodes
  
  \draw[dashed] (A) -- (B);
  \draw[dashed] (B) -- (C);
  \draw[dashed] (A) -- (D);
  \draw (D) -- (B);
  \draw[dashed] (C) -- (E);
  \draw[dashed] (D) -- (E);
  \draw[dashed] (A) -- (F);

  \filldraw[white, draw=black, thick] (A) circle (2.5pt); 
  \filldraw[black] (B) circle (2.5pt);                     
  \filldraw[white, draw=black, thick] (C) circle (2.5pt);
  \filldraw[black] (D) circle (2.5pt);
  \filldraw[white, draw=black, thick] (E) circle (2.5pt);
  \filldraw[white, draw=black, thick] (F) circle (2.5pt);
\end{scope}

\begin{scope}[xshift=4cm]
  \nodes

  \draw[dashed] (A) -- (B);
  \draw[dashed] (B) -- (C);
  \draw (A) -- (D);
  \draw[dashed] (D) -- (B);
  \draw[dashed] (C) -- (E);
  \draw[dashed] (D) -- (E);
  \draw (A) -- (F);

  \filldraw[black] (A) circle (2.5pt);
  \filldraw[white, draw=black, thick] (B) circle (2.5pt);
  \filldraw[white, draw=black, thick] (C) circle (2.5pt);
  \filldraw[black] (D) circle (2.5pt);
  \filldraw[white, draw=black, thick] (E) circle (2.5pt);
  \filldraw[black] (F) circle (2.5pt);
\end{scope}

\begin{scope}[xshift=8cm]
  \nodes

  \draw[dashed] (A) -- (B);
  \draw[dashed] (B) -- (C);
  \draw[dashed] (A) -- (D);
  \draw[dashed] (D) -- (B);
  \draw (C) -- (E);
  \draw (D) -- (E);
  \draw[dashed] (A) -- (F);

  \filldraw[white, draw=black, thick] (A) circle (2.5pt);
  \filldraw[white, draw=black, thick] (B) circle (2.5pt);
  \filldraw[black] (C) circle (2.5pt);
  \filldraw[black] (D) circle (2.5pt);
  \filldraw[black] (E) circle (2.5pt);
  \filldraw[black] (F) circle (2.5pt);
\end{scope}

\end{tikzpicture}
\end{center}
 \caption{The dynamic RCM at different times. Nodes alternate between active (full dot) and inactive (empty dot). Potential edges (dotted line) appear (solid line) when both adjacent nodes are active.}
    \label{fig:drcm}
\end{figure}

We proceed by defining the multivariate subgraph process that we focus on in this paper. 
To this end, we let  $G_i$ be a given connected graph on vertices $1,\dots,q_i$ and edge set $E(G_i)$, for $i=1,\dots,m$, and we let $a_i :=|{\rm Aut}(G_i)|$ be the cardinality of the automorphism group of $G_i$. Then we define $\Gamma_{n,i}(t)$ as the number of active subgraphs isomorphic to $G_i$ in the RCM on (the projection of $\eta_n$ onto $\R^d$), i.e., with ${\bs P}_{q_i}:= (P_1,\dots,P_{q_i})$ and ${\bs X}_{q_i}:=\{X_1,\ldots, X_{q_i}\}$,
\begin{align*}
\Gamma_{n,i}(t):&= \frac{1}{a_i} \sum_{{\bs P}_{q_i} \in \eta_{n,\neq}^{q_i}} \one\{\forall (k,\ell) \in E(G_i): X_k \leftrightarrow X_\ell \in \textrm{RCM}({\bs X}_{q_i},\vp_n)\} \cdot\prod_{k=1}^{q_i} A_k(t)
\end{align*}
where $P_k=(X_k,(A_k(t))_t)$ for $k \in [q_i]$ and $\eta_{n,\neq}^{q_i}$ is the collection of all $q_i$-tuples of pairwise distinct elements in $\eta_n$; the corresponding $m$-dimensional process is denoted by $\vk\Gamma_{n}(\cdot) = \bigl(\Gamma_{n,1}(\cdot),\ldots,\Gamma_{n,m}(\cdot)\bigr)$.

\subsection{Main result} \label{sec:main_res}
The goal of this paper is to establish a functional CLT for the $m$-dimensional process $\vk\Gamma_{n}(\cdot)$,
defined on $[0,T]$ for a given time horizon $T>0$. 
To state our result in a compact manner, we introduce some notation. 
Let $\tt G_{i,j}\HMAX$ be the collection of all pairs of graphs ${\bs H}\equiv (H_1,H_2)$ with $H_1$ being a graph on nodes $1,\dots,q_i$ isomorphic to $G_i$ and $H_2$  a graph on nodes $1,q_i+1,\dots,q_i+q_j-1$ isomorphic to $G_j$. Likewise, $\tt G_{i,j}\HMIN $ is the collection of all graph pairs ${\bs H}\equiv (H_1,H_2)$ where $H_1$ is a graph on nodes $1,\dots,q_i$ isomorphic to $G_i$ and $H_2$ is a graph on nodes $1,\dots,q_j$ isomorphic to $G_j$. Then,
\begin{align*}
    {\mathbb F}_{ij}\HMAX :=\sum_{{\bs H}\in{\tt G}_{ij}\HMAX  }\frac{F({\bs H})}{(q_i-1)!(q_j-1)!},\quad  {\mathbb F}_{ij}\HMIN:=\sum_{{\bs H}\in{\tt G}_{ij}\HMIN }\frac{F({\bs H})}{(q_i \wedge q_j)!}.
\end{align*}
Moreover, 
\begin{align}\label{eq:defF2}
F({\bs H}):= \int_{(\R^d)^{q-1}} \prod_{(u,v) \in E(H_1)\cup E(H_2)} \vp(\|y_u-y_v\|^d) \,\d(y_2,\dots,y_q),
\end{align}
where $y_1:=0$ and $q:=|V(H_1)\cup V(H_2)|$. {Denote for simplicity $F(H):=F(H,H)$.}

We now specify what conditions are imposed on $\nu_n$. In the first place,
\begin{equation}\label{eq:nu}\lim_{n\to\infty}\nu_n = 0, \quad \lim_{n\to\infty} n^{q_i} \nu_n^{q_i-1}  =\infty\quad\text{for all } i \in [m].\end{equation}
We distinguish between the {\it dense regime}  and the {\it sparse regime}, defined via
\begin{align*}
    {\mathscr D} := \{\nu_n: n \nu_n \to \infty\},\quad
    {\mathscr S} := \{\nu_n: n \nu_n \to 0\},
\end{align*}
respectively. 
For instance, picking $\nu_n:=n^{\gamma}$, condition \eqref{eq:nu} yields that $-q_i/(q_i-1)< \gamma<0$ for all $i \in [m]$. In addition, $\nu_n\in{\mathscr D}$ iff $\gamma>-1$, whereas $\nu_n\in{\mathscr S}$ iff $\gamma<-1.$
Define, for any $n\in\N$, $i\in [m]$ and $t\in[0,T]$, the centered and normalized version of $\Gamma_{n,i}(\cdot)$: with $\psi_{n,i}:=\varrho^{q_i}n^{q_i-1/2}\nu_n^{q_i-1}$ if $\nu_n\in{\mathscr D}$ and $\smash{\psi_{n,i}:=\varrho^{q_i}\sqrt{n^{q_i}\nu_n^{q_i-1}}}$ if $\nu_n\in{\mathscr S}$,
\begin{align}
    \Gamma^{\star}_{n,i}(t) :=       
    \frac{\Gamma_{n,i}(t) - \E[\Gamma_{n,i}(t)]}{\psi_{n,i}}.
    \label{gamma_star_def}
\end{align}
In addition,  with
$Z(t):=1+({\lambda}/{\mu})\,e^{-(\lambda+\mu){t} }$, let $\vk\Gamma(\cdot)$ be a centered Gaussian process with covariance matrix $\Sigma(s,t)=(\Sigma_{i,j}(s,t))_{i,j\in [m]}$ for $s,t\in[0,T]$, given by
\begin{align*}
\Sigma_{i,j}(s,t)
    &= {\rm Cov}(\Gamma_i(s),\Gamma_j(t))
     = {\rm Cov}(\Gamma_i(t),\Gamma_j(s)) \\
    &= 
    \begin{cases}
        Z(|t-s|)\,{\mathbb F}_{ij}\HMAX ,& \nu_n \in {\mathscr D},\\[4pt]
        \left(Z(|t-s|)\right)^{q_i}\,{\bs 1}\{q_i=q_j\}\,{\mathbb F}_{ij}\HMIN,& \nu_n \in {\mathscr S}.
    \end{cases}
\end{align*}

\begin{theorem}\label{main} If $\nu_n$ satisfies \eqref{eq:nu}, then $\vk\Gamma_n^*(\cdot)\to\vk\Gamma(\cdot)$  as $n\to\infty$ (in distribution in $\mathbb{D}([0,T],\R^{m})$).
\end{theorem}
\begin{remark}\label{rem_corr}
    In case $\nu_n\in\mathscr{D}$ we observe that we can write \[\mathbb{F}^{+}_{i,j} = (q_iF(G_i)/a_i)\,(q_jF(G_j)/a_j),\] implying that the components of ${\bs \Gamma}(\cdot)$ are perfectly correlated. Indeed,  ${\bs \Gamma}(\cdot)$ has the same distribution as ${\bs \Gamma}^{\prime}(\cdot)$ defined for each $i=1,\ldots,m$ as $\Gamma^{\prime}_i(\cdot) = q_iF(G_i)\xi(\cdot)/a_i$, where $\xi(\cdot)$ is centered Gaussian process with covariance function $Z(\abs{s-t})$. 
\end{remark}
 The following corollary presents an application of Theorem \ref{main}. Recall that the clustering coefficient is defined as the number of triangles divided by the number of wedges, sometimes multiplied by 3; cf.\  \cite[Eqn.\ (1.5.4)]{RvdH1}. We study a more general `subgraph ratio process':
\begin{align*}
    C_{n,G_1,G_2}(t) := \frac{a_1\,\Gamma_{n,1}(t)}{a_2\,{\Gamma_{n,2}(t)}}
\end{align*}
for connected graphs $G_1$, $G_2$ such that $\mathcal{V}(G_1) = \mathcal{V}(G_2) = q$ and {$G_2\subset G_1$}. Let $C_{G_1,G_2}(\cdot)$ denote a centered Gaussian process with covariance function 
 \begin{align*}
    \Sigma^{C}(s,t): = {\frac{a_1^2}{F^2(G_2)}\Sigma_{11}(s,t) - \frac{2a_1a_2 F(G_1)}{F^3(G_2)}\Sigma_{12}(s,t) + \frac{a_2^2F^2(G_1)}{ F^4(G_2)}\Sigma_{22}(s,t).}
 \end{align*}
 In addition, we define by $C^{\star}_{n,G_1,G_2}(\cdot)$ the centered and scaled version of the subgraph ratio process: for any $t\in[0,T]$, 
\begin{align*}
    C^{\star}_{n,G_1,G_2}(t) := \left(C_{n,G_1,G_2} - \frac{a_1}{a_2}\frac{\mathbb{E}[\Gamma_{n,1}(t)]}{{\mathbb{E}[\Gamma_{n,2}(t)]}}\right)\zeta_n,\quad\quad
    \zeta_n := \begin{cases} \sqrt{n}, & \nu_n\in{\mathscr D},\\
    \sqrt{n^q\nu_n^{q-1}}, & \nu_n\in{\mathscr S}.
    \end{cases}
\end{align*}
 As an immediate consequence of \netheo{main}, in combination with the Lemmas \ref{pF_asmpt} and \ref{lem:exp} that we establish later on, we have the following result. 

\begin{proposition}
    [Subgraph ratio process]\label{korr:clustering_coefficient} Let $G_1$, $G_2$ be two connected graphs, such that $\mathcal{V}(G_1) = \mathcal{V}(G_2) = q$ and $G_1\subset G_2$. Then, in distribution in $\mathbb{D}([0,T],\R^m)$,  we have $C^{\star}_{n,G_1,G_2}(\cdot)\to C_{G_1,G_2}(\cdot)$ as $n\to\infty$. In particular, if $\nu_n\in{\mathscr D}$, then $\Sigma^C(s,t)=0$. 
\end{proposition}

\begin{example}[Clustering coefficient process] {\em Consider a particular example of the clustering coefficient for $G_1$ being a triangle and $G_2$ being a wedge. 
Denoting
\begin{align*}
    \kappa_d := \int_{\R^d}\vp(\|{y}\|^d)\,\td y,\quad\tau_d := \int_{\left(\R^d\right)^2}\,\vp(\|{y_1}\|^d)\,\vp(\|{y_2}\|^d)\,\vp(\|y_1-y_2\|^d)\,\td (y_1,y_2),
\end{align*}
we have that $q = 3$, $a_1= 6$, $a_2 = 2$. With the compact notation ${\mathsf f}:=F({\bs H})$,
\begin{align*}
    F(G_1)& =\, \kappa_d^2,\quad F(G_2) = \tau_d,\quad \abs{{\tt G}\HMAX _{11}} =  \abs{{\tt G}\HMIN _{11}} = 9 , \quad \abs{{\tt G}\HMAX _{12}} = \abs{{\tt G}\HMIN _{12}} = 3 ,  \quad \abs{{\tt G}\HMAX _{22}} = \abs{{\tt G}\HMIN _{22}} = 1,
\end{align*}    
\begin{align*}
\mathsf{f} = 
\begin{cases}
\kappa_d^2 \tau_d & \text{for } \boldsymbol{H} \in {\tt G}\HMAX _{12} \\
\tau_d            & \text{for } \boldsymbol{H} \in {\tt G}\HMIN _{12}
\end{cases}
\quad
\mathsf{f} = 
\begin{cases}
\tau_d^2 & \text{for } \boldsymbol{H} \in {\tt G}\HMAX _{22} \\
\tau_d   & \text{for } \boldsymbol{H} \in {\tt G}\HMIN _{22}
\end{cases}
\quad
\mathsf{f} = 
\begin{cases}
\kappa_d^4 & \text{for } \boldsymbol{H} \in {\tt G}\HMAX _{11} \\
\kappa_d^2 & \text{for } \boldsymbol{H} \in {\tt G}\HMIN _{11},~ H_1 = H_2 \\
\tau_d     & \text{for } \boldsymbol{H} \in {\tt G}\HMIN _{11},~ H_1 \ne H_2
\end{cases}
\end{align*}
so that
\[
{\mathbb F}_{11}\HMAX = \frac{9\kappa_d^4}{4}, \quad 
{\mathbb F}_{11}\HMIN = \frac{3\kappa_d^2 + 6\tau_d}{6}, \quad
{\mathbb F}_{12}\HMAX = \frac{3\kappa_d^2\tau_d}{4}, \quad 
{\mathbb F}_{12}\HMIN = \frac{3\tau_d}{6}, \quad
{\mathbb F}_{22}\HMAX = \frac{\tau_d^2}{4}, \quad 
{\mathbb F}_{22}\HMIN = \frac{\tau_d}{6}.
\]
After elementary computations we find for $\nu_n \in {\mathscr D}$ that $\Sigma^C(s,t) =0$, and for $\nu_n \in {\mathscr S}$ that
\begin{align*}
    \Sigma^C(s,t) 
    &= 
    9\left(Z(\abs{t-s})\right)^3\,\Bigl(\frac{36}{\tau_d} - \frac{90\kappa^2_d}{\tau^2_d} + \frac{54\kappa^4_d}{\tau^3_d}\Bigr).& 
\end{align*}}
\end{example}

\section{Preliminaries}\label{sec:prelim}
\subsection{Cumulants}
We recall that, with $\mathrm{\bf i}$ the imaginary unit, we define the {\it cumulants} by
$$
\cum {(X_1,\,\dots,\,X_m)} := (-\mathrm {\bf i})^m \frac{\partial^m}{\partial t_1\cdots \partial t_m}\log \phi_{X_1, \dots, X_m}(t_1, \dots, t_m)\mid_{t_1 = \ldots= t_m = 0},
$$
where $\phi_{X_1, \dots, X_m}(t_1, \dots, t_m)$ is the characteristic function of the random vector $(X_1, \dots, X_m)$:
$$
\phi_{X_1, \dots, X_m}(t_1, \dots, t_m):=  \E\Big[ \exp\Big( \mathrm {\bf i}\sum_{i\le m} t_iX_i \Big) \Big] .
$$ 
Moreover, we denote the $m$-th order cumulant of a random variable $X$ by $\cum_m(X)$ and note that $\cum(X,\dots,X)=\cum_m(X)$, where the left-hand side takes an $m$-tuple as its input.

\subsection{Partitions and diagram formulae}
{Partitions} play a central role in our application of the method of cumulants. To this end, in the present subsection we reproduce some material from \cite[Section 2.4]{chaos} (see also \cite[Chapter 12.2]{lp}).

Let $m\in\N$ and let $q_1,\hdots,q_m\in\N$. We define $N_0:=0$, $\smash{N_\ell{:=}\sum_{i=1}^\ell q_i}$, $\ell\in\{1,\hdots,m\}$, and $N:=N_m$, and put $J_\ell:=\{N_{\ell-1}+1,\hdots,N_\ell\}$, $\ell\in\{1,\hdots,m\}$. 
A partition $\sigma$ of $[N]:=\{1,\hdots, N\}$ is a collection $\{B_1,\hdots,B_k\}$ of {\color{green} } $1\leq k\leq N$ {pairwise} disjoint non-empty sets, called {\it blocks}, such that $B_1\cup\ldots\cup B_k=[N]$. The number $k$ of blocks of $\sigma$ is denoted by $|\sigma|$. By $\Pi(q_1,\hdots,q_m)$ we denote the set of all partitions $\sigma$ such that $|B\cap J_\ell|\leq 1$ for all $\ell\in\{1,\hdots,m\}$ and $B\in\sigma$. For a graphical representation of partitions, see \cite[Figure 1]{chaos}.

Every partition $\sigma\in\Pi(q_1,\hdots,q_m)$ induces a {\it partition} $\sigma^*$ of $\{1,\hdots,m\}$ in the following way: $i,j\in\{1,\hdots,m\}$ are in the same block of $\sigma^*$ whenever there is a block $B\in\sigma$ such that $|B\cap J_i|=1$ and $|B\cap J_j|=1$. Let $\smash{\widetilde{\Pi}(q_1,\hdots,q_m)}$ be the set of all partitions $\sigma\in\Pi(q_1,\hdots,q_m)$ such that $|\sigma^*|=1$. By $\Pi_{\geq 2}(q_1,\hdots,q_m)$ and
$\smash{\widetilde{\Pi}_{\geq 2}(q_1,\hdots,q_m)}$ we denote the sets of all $\sigma\in\Pi(q_1,\hdots,q_m)$ and of all $\sigma\in\widetilde{\Pi}(q_1,\hdots,q_m)$ such that $|B|\geq 2$ for all $B\in\sigma$. Finally, we introduce the set $\overline{\Pi}(q_1,\ldots,q_m)$ of all partitions $\sigma\in\Pi(q_1,\ldots,q_m)$ such that for each $\ell\in\{1,\hdots,m\}$ there exists a block $B\in\sigma$ with $|B|\geq 2$ and $B\cap J_\ell \neq\varnothing$. In other words, in each row in the graphical representation of $\sigma$ there exists at least one element, which belongs to some block $B\in\sigma$ with $|B|\geq 2$. In case that $q_1=\hdots=q_m= {q}$ we simply write $\Pi^m(q)$, $\widetilde{\Pi}^m(q)$, $\widetilde{\Pi}^m_{\geq 2}(q)$ and $\overline{\Pi}^m(q)$ instead of $\Pi(q_1,\hdots,q_m)$, $\widetilde{\Pi}(q_1,\hdots,q_m)$,  $\widetilde{\Pi}_{\geq 2}(q_1,\hdots,q_m)$ and $\overline{\Pi}(q_1,\hdots,q_m)$, respectively.

For functions $f^{(\ell)}: \mathbb X^{q_\ell}\to\R$, $\ell\in\{1,\hdots,m\}$, we define their {\it tensor product} $\otimes_{\ell=1}^m f^{(\ell)}: \mathbb X^N\to\R$ by
$$
(\otimes_{\ell=1}^m f^{(\ell)})(x_1,\hdots,x_N):=\prod_{\ell=1}^m f^{(\ell)}(x_{N_{\ell-1}+1},\hdots,x_{N_\ell})\,.
$$
For $\sigma\in\Pi(q_1,\hdots,q_m)$ the function $(\otimes_{\ell=1}^m f^{(\ell)})_\sigma: \mathbb X^{|\sigma|}\to\R$ is obtained by replacing in $(\otimes_{\ell=1}^m f^{(\ell)})$ all variables that belong to the same block of $\sigma$ by a new common variable. Note that this way $(\otimes_{\ell=1}^m f^{(\ell)})_\sigma$ is only defined up to permutations of its arguments. Since in what follows we always integrate with respect to all arguments of $(\otimes_{\ell=1}^m f^{(\ell)})_\sigma$, this does not cause any problems. 

The following diagram formulae from \cite{chaos} will prove useful in our arguments. Let $m \in \mathbb N$, let $\tilde m$ with $\tilde m \ge m$  be even and let $f^{(\ell)} \in L_s^1(\mu^{q_{\ell}})$ with $q_{\ell} \in \mathbb N$, $\ell \in [\tilde m]$, be such that 
\begin{align*}
    \int_{\mathbb X^{|\sigma|}}|(\otimes_{\ell=1}^{\widetilde m} f^{(\ell)})_{\sigma}| \,\d \mu^{|\sigma|}<\infty,\qquad &\sigma \in \overline{\Pi}_{\ge 2}^{\tilde m}(q_i),\quad i \in [m],\\
    \int_{\mathbb X^{|\sigma|}}|(\otimes_{\ell=1}^{m} f^{(\ell)})_{\sigma}| \,\d \mu^{|\sigma|}<\infty,\qquad &\sigma \in \widetilde{\Pi}_{\ge 2}(q_1,\dots,q_m).
\end{align*}
Then, \cite[Theorem 3.7]{chaos} gives for $S_{\ell}:=\sum_{(p_1,\dots,p_{q_{\ell}})\in \eta_{\neq}^{q_{\ell}}} f^{(\ell)}(p_1,\dots,p_{q_{\ell}})$, $\ell \in [m]$,
\begin{align}
    \E\Big[\prod_{\ell=1}^m(S_{\ell}-\E S_{\ell})\Big] &= \sum_{\sigma \in \overline{\Pi}(q_1,\dots,q_m)}\int_{\mathbb X^{|\sigma|}}(\otimes_{\ell=1}^{m} f^{(\ell)})_{\sigma}\,\d \mu^{|\sigma|},\label{eq:mom}\\
    \text{cum}(S_1,\dots,S_m)&=\sum_{\sigma \in \widetilde{\Pi}(q_1,\dots,q_m)} \int_{\mathbb X^{|\sigma|}}(\otimes_{\ell=1}^{m} f^{(\ell)})_{\sigma}\,\d \mu^{|\sigma|}.\label{eq:cum}
\end{align}

\subsection{Two auxiliary results}
 Let $H$ be a connected graph on vertices $1,\dots,q$ with edge set $E(H)$.
The following lemma gives the asymptotic behavior of
\begin{align}
   F_n(H) := \int_{([-\frac 12,\frac 12]^d)^{q}}\prod_{(u,v)\in E(H)}\varphi_n(x_u-x_v)\td\vk x.\label{F_def}
\end{align}

\begin{lemma}\label{pF_asmpt} 
    Let $H$ be a connected graph on vertices $1,\dots,q$ with edge set $E(H)$. Then we have as $n \to \infty$,
    $$
    F_n(H) \sim \nu_n^{q-1} \int_{(\R^d)^{q-1}} \prod_{(u,v) \in E(H)} \vp(\|y_u-y_v\|^d) \d(y_2,\dots,y_q)=:\nu_n^{q-1} F(H),  
    $$
    where $y_1:=0$.
\end{lemma}

\begin{proof}
    By definition of $\vp_n$, we have
    $$
    F_n(H)= \int_{W^{q}}\prod_{(u,v)\in E(H)}\vp\Big(\frac{\|x_u-x_v\|^d}{\nu_n}\Big)\td\vk x
    $$
    Substituting $x_i:=x_1+\nu_n^{1/d}y_i,\,i =2,\dots,q$, it follows by dominated convergence that
    \begin{align*}
        F_n(H)&= \nu_n^{q-1}  \int_{(\R^d)^{q-1}} \one\Big\{x_1+\nu_n^{1/d} y_i \in \big[-\tfrac 12, \tfrac 12\big]^d\Big\} \prod_{(u,v)\in E(H)}\vp(\|y_u-y_v\|^d)\td (y_2,\dots,y_q)\\
        &\sim \nu_n^{q-1} \int_{(\R^d)^{q-1}}  \prod_{(u,v)\in E(H)}\vp(\|y_u-y_v\|^d)\td (y_2,\dots,y_q),
    \end{align*}
    where $y_1:=0$.
\end{proof}

We now state a technical lemma regarding the process $(A(t))_t$ governing the status (active/inactive) of a Poisson point. 
We leave out the proof, as it is analogous to that of \cite[Proposition 2.6]{hkm25}.

\begin{lemma}\label{edge_process_lem} There is a positive constant $\mathcal{C}$ such that for any $0\leqslant r\leqslant s\leqslant t\leqslant T$,
    \begin{align}
        \E\bigl[A(t)\bigr]& = \varrho,\label{a_exp}\\
        \E\bigl[A(s)A(t)\bigr] &= \varrho^2 Z(|{t-s}|),\label{a_cov}\\
        \mathbb{P}(A(r)\not=A(s))&\leqslant \mathcal{C}\abs{s-r},\label{a_single}\\ 
        \mathbb{P}(A(r)\not=A(s),\,A(s)\not=A(t))&\leqslant \mathcal{C}\abs{t-r}^2.\label{a_double}
    \end{align}
\end{lemma}

\subsection{Parameter specifications}
In the following arguments it will be useful to consider the $\Gamma_{n,i}$, with $i\in [m]$, as $U$-statistics. To this end, we formally enlarge our state space to\[\mathbb X:=W \times \mathbb D([0,T],\{0,1\}) \times [0,1]^{\mathbb N}.\] On $\mathbb X$ we choose the Borel measure $\mu_n:=n\, \d x \otimes \mathbb Q \otimes (\text{Unif}[0,1])^{\mathbb N}$, where $\mathbb Q$ denotes the distribution of a Markov jump process on $\{0,1\}$ with exponentially distributed waiting times with parameters $\mu$ and $\lambda$ (as described in Section 2). We write $(X_k,A_k(\cdot),(T_{k,\ell})_{\ell \ge 1})$ for a canonical random element in $\mathbb X$. Here, $A_k(t)=1$ if the point $X_k$ is active at time $t \in [0,T]$. The set of active points forms the vertex set of our RCM, where we place an edge between $X_k$ and $X_{\ell}$ if and only if
\begin{align*}
    \begin{cases}
        T_{k,\ell} \le \vp_n(X_k-X_{\ell}), \quad \text{if } X_k \prec X_{\ell},\\
        T_{\ell,k} \le \vp_n(X_k-X_{\ell}), \quad \text{if } X_{\ell} \prec X_k.
    \end{cases}
\end{align*}
where $x \prec y$ if $x$ is smaller than $y$ with respect to the lexicographic order on $\mathbb R^d$.
Let $P_{k}:=(X_{k},A_{k}(\cdot),(T_{k,\ell})_{\ell \ge 1})$ for $k \in [q_i]$ and
\begin{align}\label{def:f}
f_{n,t}^{(i)}(P_1,\dots,P_{q_i}):=\frac{1}{q_i!} \prod_{k=1}^{q_i} A_{k}(t) \cdot \sum_{H \in {\bs G}_{G_i}([q_i])} \prod_{\substack{\{k,\ell\} \in E(H)\\X_k \prec X_{\ell}}}\one\{T_{k,\ell} \le \vp_n(X_k-X_{\ell})\},
\end{align}
where ${\bs G}_{G_i}([q_i])$ is the set of all graphs on nodes $[q_i]$ that are isomorphic to $G_i$. Then $$\Gamma_{n,i}(t)=\sum_{(P_1,\dots,P_{q_i}) \in \eta_{n,\neq}^{q_i}} f_{n,t}^{(i)}(P_1,\dots,P_{q_i}).$$

\section{Expectation and covariance structure}\label{section:exp_and_cov}
In this section we subsequently consider the means and covariances pertaining to $\Gamma_{n,i}(\cdot)$.  

\begin{lemma}\label{lem:exp} 
    For any $n\in\N$, $i\in [m]$ and $t\in[0,T]$ with $a_i:=|{\rm Aut}(G_i)|$,
    \begin{align}
        \E\bigl[\Gamma_{n,i}(t)\bigr] = \frac{F_n(G_i)(\varrho n)^{q_i}}{a_i}.\label{exp}
    \end{align}
\end{lemma}
\begin{proof}
Let $f_{n,t}^{(i)}$ be as in \eqref{def:f}. Then Campbell's formula gives 
\begin{align*}
\E[\Gamma_{n,i}(t)]&= \int_{\mathbb X^{q_i}} f_{n,t}^{(i)}(p_1,\dots,p_{q_i}) \mu_n(\d (p_1,\dots,p_{q_i}))\\
&= \frac{n^{q_i}}{q_i!} \prod_{\ell \in [q_i]} \E[A_{\ell}(t)] \cdot\sum_{H \in {\bs G}_{G_i}([q_i])}  \int_{W^{q_i}}  \prod_{(k,\ell) \in E(H)} \vp_n(x_{\ell}-x_k) \d(x_1,\dots,x_{q_i}).
\end{align*}
Note that $|{\bs  G}_{G_i}([q_i])|=q_i!/|\text{Aut}(G_i)|$ and that the integral above only depends on $G_i$, but not on $H$. Since $\E[A_{\ell}(t)]=\varrho$ for all $\ell \in [q_i]$, we arrive at
$$
\frac{n^{q_i}}{|\text{Aut}(G_i)|} \varrho^{q_i}  \int_{W^{q_i}}  \prod_{(k,\ell) \in E(G_i)} \vp_n(x_{\ell}-x_k) \d(x_1,\dots,x_{q_i}),
$$
which is the assertion.
\end{proof}

\begin{lemma} \label{lem:cov}   
For any $s,t\in[0,T]$, and $i,j\in [m]$, with $F({\bs H})$ as defined in \eqref{eq:defF2},
\begin{align*}
{\rm Cov}(\Gamma_{n,i}(t),\Gamma_{n,j}(s)) \sim  \sum_{m =1} ^{q_i \wedge q_j} \sum_{H_1,H_2}  \frac{n^{q_i+q_j-m}\varrho^{q_i+q_j}Z(|t-s|)^{m}}{m! (q_i-m)! (q_j-m)!}\nu_n^{q_i+q_j-m-1} F({\bs H}),
\end{align*}
where the sum is over all graphs $H_1$ on nodes $1,\dots,q_i$ that are isomorphic to $G_i$ and all graphs $H_2$ on nodes $1,\dots,m,q_i+1,\dots,q_i+q_j-m$ that are isomorphic to $G_j$.
\end{lemma}

\begin{proof}
Let $f_{n,t}^{(i)},i\in [m],$ be as in \eqref{def:f}. We have
\begin{align*}
\E[\Gamma_{n,i}(t) \Gamma_{n,j}(s)] &= \E\Big[\sum_{(P_1,\dots,P_{q_i})\in \eta_{n,\neq}^{q_i}} \sum_{(P'_1,\dots,P'_{q_j})\in \eta_{n,\neq}^{q_j}} f_{n,t}^{(i)}(P_1,\dots,P_{q_i}) f_{n,s}^{(j)}(P'_1,\dots,P'_{q_j})\Big]\\
&=  \E\Big[\sum_{m=0}^{q_i \wedge q_j} \sum_{(P_1,\dots,P_{q_i})\in \eta_{n,\neq}^{q_i}} \sum_{(P'_1,\dots,P'_{q_j})\in \eta_{n,\neq}^{q_j}} f_{n,t}^{(i)}(P_1,\dots,P_{q_i}) f_{n,s}^{(j)}(P'_1,\dots,P'_{q_j})\\
&\hspace{4cm}\quad \times \one\{|\{P_1,\dots,P_{q_i}\}\cap \{P'_1,\dots,P'_{q_j}\}|=m\}\Big].
\end{align*}
Note that there are $\binom{q_i}{m}$ possibilities to place the $m$ common nodes among $p_1,\dots,p_{q_i}$ and $\binom{q_j}{m}$ possibilities to place them among $p'_1,\dots,p'_{q_j}$. Since there are $m!$ permutations of the $m$ common nodes, the above is given by
\begin{align*}
&\sum_{m=0}^{q_i \wedge q_j} m! \binom{q_i}{m} \binom{q_j}{m}\E\Big[\sum_{{\bs P}_{q_i}\in \eta_{n,\neq}^{q_i}} \sum_{{\bs P}'_{q_j-m}\in (\eta_{n}\setminus \{P_1,\dots,P_{q_i}\})_{\neq}^{q_j-m}} f_{n,t}^{(i)}(\bs P_{q_i}) f_{n,s}^{(j)}(\bs P_m,\bs P'_{q_j-m})\Big]\\
&\, = \sum_{m=0}^{q_i \wedge q_j} m! \binom{q_i}{m} \binom{q_j}{m}\iint\E \Big[f_{n,t}^{(i)}(\bs p_{q_i}) f_{n,s}^{(j)}(\bs p_m,\bs p'_{q_j-m})\Big] \mu_n^{q_j-m}(\d {\bs p}'_{q_j-m}) \mu_n^{q_i}(\d {\bs p}_{q_i}),
\end{align*}
where ${\bs p}_{m}=(p_1,\dots,p_{m})$ and ${\bs p}'_{q_j-m}=(p'_1,\dots,p'_{q_j-m})$ (and analogously for capital and bold letters). For each $m=0,\dots,q_i \wedge q_j$,  the above integral is given by
\begin{align} 
\frac{n^{q_i+q_j-m} \varrho^{q_i+q_j}Z(|t-s|)^{m}}{q_i! q_j!} \sum_{H_1,H_2} \int_{W^{q_i+q_j-m}} \prod_{(k,\ell) \in E(H_1) \cup E(H_2)} \vp_n(x_k-x_{\ell}) \d(x_1,\dots,x_{q_i+q_j-m}),
\end{align} 
where the sum is over all graphs $H_1$ on nodes $1,\dots,q_i$ that are isomorphic to $G_i$ and all graphs $H_2$ on nodes $1,\dots,m,q_i+1,\dots,q_i+q_j-m$ isomorphic to $G_j$. 

Note that the contribution from $m=0$ is precisely $\E[\Gamma_{n,i}(t)] \E[\Gamma_{n,j}(s)]$. This gives for the covariance of $\Gamma_{n,i}(t)$ and $\Gamma_{n,j}(s)$,
$$
\sum_{m=1}^{q_i \wedge q_j}\frac{n^{q_i+q_j-m} \varrho^{q_i+q_j} Z(|t-s|)^{m}}{m! (q_i-m)! (q_j-m)!} \sum_{H_1,H_2} \int_{W^{q_i+q_j-m}} \prod_{(k,\ell) \in E(H_1) \cup E(H_2)} \vp_n(x_k-x_{\ell}) \d(x_1,\dots,x_{q_i+q_j-m}).
$$
Since the graph $G([q_i+q_j-m],E(H_1) \cup E(H_2))$ is connected for all $m=1,\dots,q_i \wedge q_j$, the assertion follows from Lemma \ref{pF_asmpt}.

\end{proof}

\section{Finite-dimensional convergence}\label{section:fdd_convergence}

\begin{proposition}\label{FDD_conv}
Fix $0\le t_1<\cdots <t_k \le T$ for some $k \in \mathbb N$. Then $
(\vk \Gamma^{\star}_n(t_1), \dots,  \vk\Gamma^{\star}_n(t_k))
$
converges to a Gaussian vector as $n \to \infty$.
\end{proposition}

\begin{proof}
 By the Cram\'er-Wold device, it suffices to show that, for any vectors $\vk b_1,\dots, \vk b_h\in \mathbb R_+^m$, the random variable $\smash{S_n:=\sum_{i=1}^h \langle \vk b_i, \vk \Gamma_{n}^{\star}(t_i) \rangle }$ converges to a normally distributed random variable. Our idea is to write $S_n$ as a sum of Poisson $U$-statistics and to apply the diagram formulae from Section \ref{sec:prelim} to show that the cumulants $\mathsf{cum}_M(S_n),\, M \ge 3$, converge to 0 as $n \to \infty$.  We have, with $\vk b_i=(b_{i1},\dots,b_{im}),\,i \in [h],$
 \begin{align*}
     S_n&=\sum_{i=1}^h \langle \vk b_i, \vk \Gamma_{n}^{\star}(t_i) \rangle =\sum_{i=1}^h \sum_{j=1}^m \frac{b_{i,j}}{\psi_{n,j}}(\Gamma_{n,j}(t_i)-\E \Gamma_{n,j}(t_i))\\
     &= \sum_{j=1}^m \frac{1}{\psi_{n,j}}\sum_{i=1}^h b_{i,j}(\Gamma_{n,j}(t_i)-\E \Gamma_{n,j}(t_i)).
 \end{align*}
 For any $j\in [m]$, we let 
 $$
S_n^{(j)}:=\sum_{i=1}^h b_{i,j}\Gamma_{n,j}(t_i),
 $$
 and note that, by the multilinearity of the cumulant,
\begin{align}
 \cum_M(S_n)=\cum(S_n,\dots,S_n)=\sum_{(j_1,\dots,j_M) \in [m]^M}\frac{1}{\psi_{n,j_1}\cdots \psi_{n,j_M}}\cum(S_n^{(j_1)},\dots,S_n^{(j_M)}).\label{eq:cumM}
\end{align}
 For $P_{k}=(X_{k},A_{k}(\cdot),(T_{k,\ell})_{\ell \ge 1})\in W \times \mathbb D([0,T],\{0,1\}) \times [0,1]^{\mathbb N}  ,\,k \in [q_j]$, let 
$$
 g_n^{(j)}(P_1,\dots,P_{q_j}):=\frac{1}{q_j!} \sum_{H \in \mathcal G_{G_j}([q_j])} \prod_{(k,\ell)\in E(H)} \one\{X_k \leftrightarrow X_\ell \text{ in }\textrm{RCM}({\bs X}_{q_j}, \vp_n)\} \sum_{i=1}^h b_{ij} \prod_{\ell=1}^{q_j} A_\ell(t_i),
$$
and note that
$$
S_n^{(j)}= \sum_{(P_1,\dots,P_{q_j})\in \eta_{n,\neq}^{q_j}} g_n^{(j)}(P_1,\dots,P_{q_j}),\quad j=1,\dots,m.
$$
By \eqref{eq:cum}, for all $(j_1,\dots,j_M) \in [m]^M$,
\begin{align}
\text{cum}(S_n^{(j_1)},\dots,S_n^{(j_M)})=\sum_{\sigma \in \widetilde{\Pi}(q_{j_1},\dots,q_{j_M})}\int_{\mathbb X^{|\s|}}  (\otimes_{s=1}^M g_n^{(j_{s})})_\s \d\mu_n^{|\s|}.\label{eq:cum_fidi}
\end{align}
Before estimating \eqref{eq:cum_fidi}, we explain how to define an auxiliary graph on $\{1,\hdots,|\s|\}$ with edge set $E_{\s}$ based on a $\sigma\in\Pi(q_{j_1},\hdots,q_{j_M})$: identify the block $B_k \in \sigma$ with node $k=1,\dots, |\sigma|$ and put an edge $(k,\ell)\in E_{\s}$, if there are $k_i,k_j\in [q_{j_s}]$ for some $s \in [M]$, such that $N_{s-1} +i_k\in B_{k}$, $N_{s-1} + i_{\ell}\in B_{\ell}$ and there is an edge between $i_k$ and $i_{\ell}$ in $G_{j_s}$. For instance, let $m=M=2$, $j_1=1$, $j_2=2$ and $G_1=G_2$ be the graph on vertices $\{1,2,3\}$ with edges $(1,2)$ and $(2,3)$, let $\sigma=\{B_1,\dots,B_5\}$ with $B_1=\{1,4\}$, $B_2=\{2\}$, $B_3=\{3\}$, $B_4=\{5\}$, $B_5=\{6\}$. Then $V=\{1,\dots,5\}$ and $E_{\s}=\{(1,2), (2,3), (1,4), (4,5)\}$. 
Note, moreover, that in this case 
$$
(f^{(1)} \otimes f^{(2)})_{\sigma}(p_1,\dots,p_5)= f^{(1)}(p_1,p_2,p_3) f^{(2)}(p_1,p_4,p_5), \quad p_1,\dots,p_5 \in \mathbb X.
$$

Let $\mathcal M_s:=\{k \in [|\s|]:\,J_s \cap B_k \neq \emptyset\},\, s \in [M]$. In our context, the introduction of the auxiliary graph described above yields that for each $(j_1,\dots,j_M) \in [m]^M$ and $\sigma \in \widetilde{\Pi}(q_{j_1},\dots,q_{j_M})$, the integral in \eqref{eq:cum_fidi} is given by 
\begin{align}
n^{|\s|} \E\Big[\prod_{s =1}^M \sum_{i=1}^h b_{ij_{s}} \prod_{k \in \mathcal M_s} A_k(t_i)\Big]\int_{W^{|\s|}}  \prod_{(k,\ell)\in E_{\s}}  \vp_n(x_k - x_{\ell}) \,\d\xx_{|\s|}.\label{eq:int_graphs}
\end{align}
{
	We shall give a few more details regarding the above equation: In \eqref{eq:cum_fidi}, the integration over $\mathbb X^{|\s|}$ formally is an integration over $(W\times \mathbb D([0,T],\{0,1\})\times[0,1]^{\mathbb N})^{|\s|}$ and the integral is split accordingly. The integrand then is the indicator that for an edge $(k,\ell)\in E_{\s}$ (the edge set of the graph induced by $\s$), there is an edge in the RCM between the points $P_k$ and $P_{\ell}$. The latter now depends on the marks $T_{k,\ell}$ as described above. The integral over $[0,1]^{|\s|}$ is then formally solved by writing the probability that $T_{k,\ell}\le\varphi_n(\cdot)$ as $\varphi_n(\cdot)$. Further, the expectation in the right-hand side of \eqref{eq:int_graphs} represents the integral over $[1,\infty)^{|\s|}$ so that the only ``visible'' integral is the one with respect to the Lebesgue measure $n^{|\s|}\d(x_1,\dots,x_{|\s|})$.  
}

The expectation in \eqref{eq:int_graphs} is handled as follows:
\begin{align*}
 \E\Big[\prod_{s =1}^M \sum_{i=1}^h b_{ij_{s}} \prod_{k \in \mathcal M_s} A_k(t_i)\Big]&=\sum_{i_1,\dots,i_M=1}^h \Big(\prod_{s =1}^M b_{i_{s}j_s}\Big) \E \Big[ \prod_{k \in \mathcal M_s} \prod_{s=1}^M A_k(t_{i_s})\Big]\\
    &\le\sum_{i_1,\dots,i_M=1}^h \Big(\prod_{s =1}^M b_{i_{s}j_s}\Big)  \E \Big[ \prod_{k\in \mathcal M_s}  A_k(t_{i_{s(k)}}) \Big],
\end{align*}
where $s(k)\in [M]$ is the minimum index $s$ with $J_{s} \cap B_k \neq \emptyset$. By stationarity of $(A_k(t))_{t\ge0}$ for all $k\in [|\s|]$ and independence of $(A_1(t))_{t>0},\dots,(A_{|\s|}(t))_{t>0}$, the above is given by
$$
\sum_{i_1,\dots,i_M=1}^h \Big(\prod_{s =1}^M b_{i_{s}j_s}\Big) \varrho^{|\s|} \le \varrho^{|\s|}  (h b_{\max})^M,
$$
where we recall that $\varrho =\E[A_k(0)]$ and $b_{\max}:=\max_{i,j} b_{ij}$.

To bound the integral over $W^{|\s|}$ in \eqref{eq:int_graphs}, we argue as for \cite[(6.2)]{heerten}. This yields
$$
\text{cum}(S_n^{(j_1)},\dots,S_n^{(j_M)}) \le |\widetilde{\Pi}(q_{j_1},\dots,q_{j_M})|\varrho^{|\s|}  (h b_{\max})^M n^{|\sigma|} \nu_n^{|\sigma|-1}.
$$
We next discuss the resulting bounds for $\cum_M(S_n)$ separately for the sparse and the dense regime. If $\nu_n \in \mathscr S$, then by \eqref{eq:cumM}, the definition of $\psi_{n,j}$ and since $|\s|\ge \min\{q_1,\dots,q_m\}=:p$,
\begin{align*}
    \cum_M(S_n) = \sum_{(j_1,\dots,j_M) \in [m]^M}\frac{\cum(S_n^{(j_1)},\dots,S_n^{(j_M)}) }{\psi_{n,j_1}\cdots \psi_{n,j_M}} \le C \frac{n^{|\s|}\nu_n^{|\s|-1} }{n^{Mp/2} \nu_n^{M(p-1)/2}}\le C (n^{p}\nu_n^{p-1})^{1-M/2},
\end{align*}
which tends to 0 as $n\to \infty$ for all $M \ge 3$. 

If $\nu_n \in \mathscr D$, then since $|\s| \le \sum_{i \in [M]}(q_{j_i}-1)+1$, with $Q_M:=\sum_{i \in [M]}q_{j_i}$,
\begin{align*}
    \cum_M(S_n) &= \sum_{(j_1,\dots,j_M) \in [m]^M}\frac{\cum(S_n^{(j_1)},\dots,S_n^{(j_M)}) }{\psi_{n,j_1}\cdots \psi_{n,j_M}} \le \sum_{(j_1,\dots,j_M)\in [m]^M} C \frac{n^{Q_M-M+1}\nu_n^{Q_M-M}}{n^{Q_M-M/2}\nu_n^{Q_M-M}}\\&\le Cn^{1-M/2},
\end{align*}
which also tends to 0 as $n\to \infty$ for all $M \ge 3$.
\end{proof}

\section{Tightness}\label{section_tightness}
Finally, to justify the claim of \netheo{main}, we apply the analogue of \cite[Theorem 13.5]{billingsley2013convergence} for vector-valued processes; in particular, the sufficient condition given in \cite[Eqn.\ (13.14)]{billingsley2013convergence}. According to this condition, {to prove that the sequence {$\vk \Gamma^{\star}_n(\cdot)$} is tight in $\mathbb{D}([0,T],\R^{m})$}, it suffices to show that for any {$0\leqslant r\leqslant s\leqslant t\leqslant T$,} 
\begin{align}   \E\left[\left\lvert\left\lvert\vk\Gamma^{\star}_n(r) - \vk\Gamma^{\star}_{n}(s)\right\rvert\right\rvert^2\left\lvert\left\lvert\vk\Gamma^{\star}_n(s) - \vk\Gamma^{\star}_{n}(t)\right\rvert\right\rvert^2\right]\leqslant C(t-r)^2\label{Tightness}
\end{align}
for some positive constant $C$.
Define, for any $n\in\N$, $i,j\in[m]$, $r,s,t\in[0,T]$,
\begin{align*}
    \Delta_{n,i,j}(r,s,t) := \E\biggl[\bigr(\Gamma_{n,i}(r) - \Gamma_{n,i}(s)\bigl)^2\bigr(\Gamma_{n,j}(s) - \Gamma_{n,j}(t)\bigl)^2\biggr],
\end{align*}
so that, due to \eqref{gamma_star_def}, as $n\nu_n \to \infty$,
\begin{align}   &\E\left[\left\lvert\left\lvert\vk\Gamma^{\star}_n(r) - \vk\Gamma^{\star}_{n}(s)\right\rvert\right\rvert^2\left\lvert\left\lvert\vk\Gamma^{\star}_n(s) - \vk\Gamma^{\star}_{n}(t)\right\rvert\right\rvert^2\right]\nonumber\\
&\quad= \sum_{i,j=1}^{m}\varrho^{-2(q_i+q_j)}\Delta_{n,i,j}(r,s,t) \begin{cases} (n^{2q_i + 2q_j-2}\nu_n^{2q_i+ 2q_j-4})^{-1}, & \nu_n\in{\mathscr D},\\
   (n^{q_i + q_j}\nu_n^{q_i + q_j-2})^{-1}, & \nu_n\in{\mathscr S}.
    \end{cases}\label{delta}
\end{align}
We define the (symmetric) function
$$
f^{(i)}(P_1,\dots,P_{q_i})=\frac{1}{q_i!} \Big(\prod_{k=1}^{q_i} A_{k}(r)-\prod_{k=1}^{q_i} A_{k}(s)\Big)\sum_{H \in {\bs G}_{G_i}([q_i])} \prod_{\substack{\{k,\ell\} \in E(H)\\X_k \prec X_{\ell}}} \hspace{-3pt}\one\{T_{k,\ell} \le \vp_n(X_k-X_{\ell})\}
$$
(and analogously $f^{(j)}$) and note that
\begin{align*}
    \Gamma_{n,i}(r) - \Gamma_{n,i}(s)=  \sum_{(P_1,\dots,P_{q_i})\in \eta_{n,\neq}^{q_i}} f^{(i)}(P_1,\dots,P_{q_i}),
\end{align*}
which is a $U$-statistic with mean zero (due to stationarity of the $A_k(\cdot),\, k \in [q_i]$), and an analogous representation holds for $\Gamma_{n,j}(r) - \Gamma_{n,j}(s)$. Hence, \cite[Theorem 3.7]{chaos} gives
\begin{align} \label{eq:intotimes}
\Delta_{n,i,j}(r,s,t)=\sum_{\s \in \ov{\Pi}(q_i,q_i,q_j,q_j)} \int_{\mathbb X^{|\s|}} \big(f^{(i)}\otimes f^{(i)}  \otimes f^{(j)} \otimes f^{(j)}\big)_{\s} \d \mu_n^{|\s|}.
\end{align}

Analogously to Section \ref{section:fdd_convergence}, we define an auxiliary graph on nodes $[|\s|]$ with edge set $E_{\s}$ for $\s \in \ov{\Pi}(q_i,q_i,q_j,q_j)$. Then, the integral in \eqref{eq:intotimes} is given by
\begin{align}
\frac{n^{|\s|}}{(q_i!)^2(q_j!)^2 } \E\Big[\prod_{\ell=1}^4 \Big(\prod_{h: J_{\ell} \cap B_h \neq \emptyset} \hspace{-.3cm} A_h(t_{\ell}) -\hspace{-.3cm} \prod_{h: J_{\ell} \cap B_h \neq \emptyset} \hspace{-.3cm} A_h(t'_{\ell}) \Big) \Big] \int_{W^{|\s|}} \prod_{(k,\ell)\in E_{\s}} \vp_n(x_k-x_{\ell}) \d(x_1,\dots,x_{|\s|}),\label{eq:tight}
\end{align}
where $t_1=t_2=r$, $t_3=t_4=t'_1=t'_2=s$ and $t'_3=t'_4=t$. Next, we bound the expectation and the integral above separately. For the expectation, we use that due to $A_h(\cdot) \in \{0,1\},\,h =1,\dots,|\s|$, 
\begin{align}
\E\Big[\prod_{\ell=1}^4 \Big(\prod_{h: J_{\ell} \cap B_h \neq \emptyset} A_h(t_{\ell}) -\prod_{h: J_{\ell} \cap B_h \neq \emptyset} A_h(t'_{\ell}) \Big) \Big]  &\le \E\Big[\prod_{\ell=1}^4 \sum_{h: J_{\ell} \cap B_h \neq \emptyset}  |A_h(t_{\ell})-A_h(t'_{\ell})| \Big]\nonumber\\
&\hspace{-3cm}= \sum_{h_1,h_2,h_3,h_4=1}^{|\s|} \E\Big[\prod_{\ell=1}^4 \one\{J_{\ell} \cap B_{h_{\ell}} \neq \emptyset\} |A_{h_{\ell}}(t_{\ell}) - A_{h_{\ell}}(t'_{\ell})|\Big].\label{eq:tight_exp}
\end{align}
For the latter expectation to be non-zero, we must have $h_1=h_2$ and $h_3=h_4$, or $h_1=h_3$ and $h_2=h_4$, or $h_1=h_4$ and $h_2=h_3$. If $h_1=h_2$ and $h_3=h_4$, we have by \eqref{a_single},
\begin{align*}
\E\Big[\prod_{\ell=1}^4 \one\{J_{\ell} \cap B_{h_{\ell}} \neq \emptyset\} |A_{h_{\ell}}(t_{\ell}) - A_{h_{\ell}}(t'_{\ell})|\Big] &\le \E[|A_{h_1}(r) - A_{h_{1}}(s)|^2] \E[|A_{h_3}(s) - A_{h_{3}}(t)|^2]\\ &\le \mathcal{C}^2 |r-s||s-t|.
\end{align*}
If $h_1=h_3$ and $h_2=h_4$, or $h_1=h_4$ and $h_2=h_3$, we find by \eqref{a_double},
\begin{align*}
\E\Big[\prod_{\ell=1}^4 \one\{J_{\ell} \cap B_{h_{\ell}} \neq \emptyset\} |A_{h_{\ell}}(t_{\ell}) - A_{h_{\ell}}(t'_{\ell})|\Big] &\le \E[(|A_{h_1}(r) - A_{h_{1}}(s)||A_{h_1}(s) - A_{h_{1}}(t)|)^2]\\ &\le \mathcal{C}^2 |r-t|^2.
\end{align*}
We conclude that \eqref{eq:tight_exp} is bounded by $|\s|^4 \mathcal{C}^2 |r-t|^2$.

In order to deal with the integral in \eqref{eq:tight}, we distinguish between the scenarios where the graph generated by $\s$ is connected or not. Since $\s \in \overline{\Pi}(q_i,q_i,q_j,q_j)$, for each $\ell\in\{1,\hdots,4\}$ there exists a block $B\in\sigma$ with $|B|\geq 2$ and $B\cap J_\ell \neq\varnothing$. This implies that the graph generated by $\s$ is either connected or has exactly two connected components. If it is connected, we obtain by Lemma \ref{pF_asmpt} that there is some constant $\beta_1=\beta_1(G_i,G_j)>0$ such that
$$
I_n(\s):=\int_{W^{|\s|}} \prod_{(k,\ell)\in E_{\s}} \vp_n(x_k-x_{\ell}) \d(x_1,\dots,x_{|\s|}) \sim \beta_1 \nu_n^{|\s|-1}.
$$
If instead the graph generated by $\s$ is disconnected, then it has exactly two connected components. Then we can decompose it into two disjoint connected subgraphs with vertex set $E'$ and $E''$, say. In this case we apply Lemma \ref{pF_asmpt} twice, and obtain for a constant $\beta_2>0$ that 
\begin{align}
  I_n(\s)= \int_{W^{|\s|}} \Big(\prod_{(k,\ell)\in E'} \vp_n(x_k-x_{\ell}) \Big) \Big(\prod_{(k,\ell)\in E''} \vp_n(x_k-x_{\ell}) \Big)\d(x_1,\dots,x_{|\s|}) \sim\beta_2 \nu_n^{|\s|-2}.
\end{align}
Note that if the graph generated by $\s$ is connected, then $|\s|\in \{\max(q_i,q_j),2q_i+2q_j-3\}$, and if it has two connected components, then $|\s|\in \{q_i+q_j,2q_i+2q_j-2\}$. This gives for $n$ large enough,
\begin{align}\label{eq:tight2}
    &n^{|\s|}I_n(\s)\le \begin{cases} \beta_2 n^{2q_i+2q_j-2}\nu_n^{2q_i+2q_j-4},\quad & \nu_n\in{\mathscr D},\\    \max(\beta_1n(n\nu_n)^{\max(q_i,q_j)-1},\beta_2n^2(n\nu_n)^{q_i+q_j-2}), \quad &\nu_n\in{\mathscr S}.
    \end{cases}
\end{align}
The assumption $n(n\nu_n)^{q_i-1} \to \infty$, $i \in [m]$, guarantees that the maximum in the second case is attained at $\beta_2n^2(n\nu_n)^{q_i+q_j-2}$. Thus, the $n$- and $\nu_n$-terms in \eqref{delta} cancel out with those from \eqref{eq:tight2}, which establishes the desired bound in \eqref{Tightness}.

\section{Subgraph ratio process}
\proofprop{korr:clustering_coefficient}
    For any $t\in[0,T]$ and $n\in\N$ random variable $C_{n,G_1,G_2}(t)$ is well-defined everywhere except on the event $\{\Gamma_{n,1} = 0\}\cap\{\Gamma_{n,2} = 0\}$. Bearing in mind that
    \[ 
        \lim_{n\to\infty}\pk{\inf_{t\in[0,T]}\Gamma_{n,1}(t) = 0,\, \inf_{t\in[0,T]}\Gamma_{n,2}(t) = 0}=0,\]
        the stated functional asymptotic Gaussianity is a direct application of the functional version of the delta method. {We now apply \netheo{main} with $\psi^\circ_n:=  \nu_n^{q-1}(\varrho n)^{q}$, and recall that,  using Lemmas \ref{pF_asmpt} and \ref{lem:exp}, $\mu_{n,i}(t) = \psi^\circ_n {F(G_i)}/{a_i}+o(\psi^\circ_n)$ as $n\to\infty$ for both $i=1,2$ uniformly in $t$. It thus follows that
        \begin{align*}
            \frac{\psi^{\circ}_n}{\psi_n}\left(\frac{\vk\Gamma_n(\cdot)}{\psi^\circ_n} - \vk v\right)\to \vk\Gamma(\cdot)
        \end{align*}
        as $n\to\infty$, in distribution in $\mathbb{D}([0,T],\R^{m})$, where we denote $\vk v := (F(G_1)/a_1,\,F(G_2)/a_2)^{\top}$. Hence, applying the delta method to the function $h(\vk x) := a_1x_1/a_2x_2$, as $n\to\infty$,
        \begin{align*}
            \frac{\psi_n^{\circ}}{\psi_n}\biggl(h\left(\frac{\vk\Gamma_n(\cdot)}{\psi_n^{\circ}}\right) - h\left(\vk v\right)\biggr)\to \vk\Gamma^{\circ}(\cdot)
        \end{align*}
        in distribution in $\mathbb{D}([0,T],\R^{m})$, where $\vk\Gamma^{\circ}(\cdot)$ is a centered Gaussian process with covariance function
        \begin{align*}
            \Sigma^{\circ}(s,t) = \nabla h\left(\vk v\right)^{\top}\Sigma(s,t)\,\nabla h\left(\vk v\right).
        \end{align*}
        Straightforward calculations show that $\Sigma^{\circ}(\cdot,\cdot) = \Sigma^{C}(\cdot,\cdot)$, and in case $\nu_n\in\mathscr{D}$, due to Remark \ref{rem_corr}, we find $\Sigma^{\circ}(s,t) = 0$ for any $s,t$, so that $\vk\Gamma^{\circ}(\cdot) = C_{G_1,G_2}(\cdot)$ in distribution. This establishes the claim as  $\psi_n^\circ/\psi_n=\zeta_n$ and $C_{n,G_1,G_2}(\cdot) = h(\vk \Gamma_n(\cdot)) = h(\vk \Gamma_n(\cdot)/\psi^{\circ}_n)$.\qed}

\bigskip

\bibliographystyle{abbrv}
\bibliography{lit}
\end{document}